 \def\misajour{16/09/2013} 
\documentclass[10pt]{article} 
 
\usepackage[applemac]{inputenc}
\usepackage[pdftex]{graphicx}

\usepackage{amsmath} 

\usepackage{amsthm} 

\usepackage[applemac]{inputenc}
\usepackage[pdftex]{graphicx} 
\usepackage{hyperref} 
 
\def\borne{M}
\def\C{\mathbf {C}}
\def\Q{\mathbf{ Q}}
\def\R{\mathbf {R}}
\def\Z{\mathbf {Z}}
\def\ZK{\Z_K}

\def\ZKtimes{\ZK^\times}

\def\Card{{\mathrm{Card}}}

\def\thetanu{\nu}

\def\ulambda{\underline{\lambda}} 

\def\Atilde{\tilde{A}}
\def\Btilde{\tilde{B}}
\def\Dtilde{\tilde{D}}

\def\house#1{\setbox1=\hbox{$\,#1\,$}%
\dimen1=\ht1 \advance\dimen1 by 2pt \dimen2=\dp1 \advance\dimen2 by 2pt
\setbox1=\hbox{\vrule height\dimen1 depth\dimen2\box1\vrule}%
\setbox1=\vbox{\hrule\box1}%
\advance\dimen1 by .4pt \ht1=\dimen1
\advance\dimen2 by .4pt \dp1=\dimen2 \box1\relax}

\def\rmN{\mathrm{N}}
\def\calH{{\mathcal{H}}}
\def\calE{{\mathcal{E}}}
\def\calEtilde{\tilde{\calE}}
\def\calF{{\mathcal{F}}}
\def\rmM{\mathrm{M}}
\def\rmh{\mathrm{h}}
\def\tors{{\mathrm{tors}}}
\newtheorem{conjecture}
{\indent Conjecture}
\newtheorem{theoreme}
{\indent Theorem}
\newtheorem{lemme}
{\indent Lemma}
\newtheorem{proposition}
{\indent Proposition}
\newtheorem{corollaire}
{\indent Corollary}
\newtheorem*{remarque}{\indent Remark}

\newcounter{compteurkappa} 

\def\Newcst#1{
\refstepcounter{compteurkappa}
\kappa_{ 
\arabic{compteurkappa}}
\label{#1}
}

\def\cst#1{\kappa_{\ref{#1}}}

\def\boxit#1#2{\setbox1=\hbox{\kern#1{#2}\kern#1}%
\dimen1=\ht1 \advance\dimen1 by #1 \dimen2=\dp1 \advance\dimen2 by #1
\setbox1=\hbox{\vrule height\dimen1 depth\dimen2\box1\vrule}%
\setbox1=\vbox{\hrule\box1\hrule}%
\advance\dimen1 by .4pt \ht1=\dimen1
\advance\dimen2 by .4pt \dp1=\dimen2 \box1\relax}

\begin{document}
 
 \hfill 
 
 \null
 \vskip -3 true cm

 \hfill
 {\it \misajour}
 
 \bigskip 

\begin{center}

{\Large
\bf 
 Solving 
 effectively some families
 \\
 of Thue Diophantine equations 

\bigskip
\it Claude Levesque \rm and \it Michel Waldschmidt
}
\end{center}

\section*{Abstract} Let $\alpha$ be an algebraic number of degree $d\ge 3$ and let $K$ be the algebraic number field $\Q(\alpha)$.
When $\varepsilon$ is a unit of $K$ such that $\Q(\alpha\varepsilon)=K$, we consider the irreducible polynomial $f_\varepsilon(X) \in \Z[X]$ such that 
 $f_\varepsilon(\alpha\varepsilon)=0$. Let $F_\varepsilon(X,Y)$ be the irrreducible binary form of degree $d$ associated to $f_{\varepsilon}(X) $
under the condition $F_{\varepsilon}(X,1)=f_{\varepsilon}(X)$. For each positive integer $m$, we want to exhibit an effective upper bound for the solutions $(x,y,\varepsilon)$ of the diophantine inequation $|F_\varepsilon(x,y)|\le m$. We achieve this goal by restricting ourselves to a subset of units $\varepsilon$ 
which we prove to be sufficiently large as soon as the degree of $K$ is $\geq 4$. 

\medskip
\noindent
{\small AMS Classification: Primary 11D61 
		Secondary 11D41, 11D59}
 
\section{The conjecture and the main result }

Let $\alpha$ be an algebraic number of degree $d\ge 3$ over $\Q$. We denote by $K$ the algebraic number field $\Q(\alpha)$, by $f\in \Z[X]$ the irreducible polynomial of $\alpha$ over $\Z$, by $ \ZKtimes$ the group of units of $K$ and by $r$ the rank of the abelian group $\ZKtimes$. For any unit $\varepsilon\in \ZKtimes$ such that the degree $\delta=[\Q(\alpha\varepsilon):\Q]$ be $\geq 3$, we denote by $f_{\varepsilon}(X)\in \Z[X]$ the irreductible polynomial of $\alpha\varepsilon$ over $\Z$ (uniquely defined upon requiring that the leading coefficient be $>0$) and by $F_{\varepsilon}$ the irreductible binary form defined by 
 $F_{\varepsilon}(X,Y) = Y^\delta f_\varepsilon(X/Y)\in\Z[X,Y]$. \\

The purpose of this paper is to investigate the following conjecture.

\begin{conjecture}
\label{ConjecturePrincipale}
There exists an effectively computable constant $\Newcst{kappa1}>0$, depending only upon $\alpha$, such that, for any $m\ge 2$, each solution $(x,y,\varepsilon)\in\Z^2\times\ZKtimes$ of the inequation $| F_\varepsilon(x,y)|\le m$ with $xy\neq 0$ and $[\Q(\alpha\varepsilon) : \Q]\ge 3$ verifies
$$
 \max\{|x|,\; |y|,\; e^{\rmh(\alpha\varepsilon)}\}\le m^{\cst{kappa1}}.
$$
\end{conjecture}

We noted $\rmh$ the absolute logarithmic height (see {\rm{($\ref{Equation:hauteur}$)}} below). \\

To prove this conjecture, it suffices to restrict ourselves to units $\varepsilon$ of $K$ such that $\Q(\alpha\varepsilon)=K$: as a matter of fact, the field $K$ has but a finite number of subfields.
An equivalent formulation of the conjecture $\ref{ConjecturePrincipale}$ is then the following one: {\it if $xy\not=0$ and $\Q(\alpha\varepsilon)=K$, then 
 $$ 
|\rmN_{K/\Q} (x-\alpha\varepsilon y)|
 \ge \Newcst{kappa2} \max\{|x|,\; |y|,\; e^{\rmh(\alpha\varepsilon)}\}^{\Newcst{kappa3}}
$$
with effectively computable positive constants $\cst{kappa2}$ and $\cst{kappa3}$, depending only upon $\alpha$}.\\

The finiteness of the set of solutions $(x,y,\varepsilon)\in\Z^2\times\ZKtimes$ of the inequation $| F_\varepsilon(x,y)|\le m$ with $xy\neq 0$ and $[\Q(\alpha\varepsilon) : \Q]\ge 3$ follows from Corollary 3.6 of $\cite{LW1}$ (which deals with Thue--Mahler equations, while in this paper we restrict ourselves to Thue equations). 
The proof in $\cite{LW1}$ rests on Schmidt's subspace theorem; it allows to exhibit explicitly an upper bound for the number of solutions as a function of $m$, $d$ and the height of $\alpha$, but it does not allow to give an upper bound for the solutions. The particular case of the conjecture $\ref{ConjecturePrincipale}$, in which the form $F$ is of degree $3$ and the rank of the unit group of the cubic field $\Q(\alpha)$ is $1$, was taken care of in $\cite{LW2}$. In $\cite{LW3}$, we 
 considered a slightly more general case, namely when the number of real embeddings of $K$ into $\C$ is $0$ or $1$, while restricting to units $\varepsilon$ such that $\Q(\alpha\varepsilon)=K$. In this paper, we prove that the conjecture is true at least for a subset $\calEtilde_\nu^{(\alpha)}$ of units, the definition of which is given in the following.\\

Denote by $\Phi=\{\sigma_1,\ldots,\sigma_d\}$ the set of embeddings of $K$ into $\C$ and by $\house{\gamma}$ the {\it house} of an algebraic 
number $\gamma$, defined to be the maximum of the moduli of the Galois conjuguates of $\gamma$ in $\C$. In symbols, for $\gamma\in K$, 
$$
\house{\gamma}=\max_{1\le i\le d} |\sigma_i(\gamma)|.
$$
The {\it absolute logarithmic height} is noted $\rmh$ 
 and involves the {\it Mahler measure} $\rmM$: 
\begin{equation}\label{Equation:hauteur}
\rmh(\alpha)=\frac{1}{d}\log \rmM(\alpha)
\quad
\hbox{with}
\quad
\rmM(\alpha)=
a_0
 \prod_{1\le i\le d} \max\{1,|\sigma_i(\alpha)|\},
\end{equation}
$a_0$ being the leading coefficient of the irreducible polynomial of $\alpha$ over $\Z$.\\

The set 
$$
\calE^{(\alpha)}=\{\varepsilon\in \ZKtimes\; \mid \; \Q(\alpha\varepsilon)=K\}
$$
depends only upon $\alpha$; (we have supposed $\Q(\alpha)=K$). When $\nu$ is a real number in the interval 
$]0,1[$, 
we denote by $\calE_\nu^{(\alpha)}$ the set of units $\varepsilon\in\calE^{(\alpha)}$ for which there exist two distinct elements $\varphi_1 $ and $\varphi_2 $ of $\Phi$ such that 
$$
|\varphi_1(\alpha\varepsilon)|= \house{\alpha\varepsilon}
\quad \hbox{and}\quad
 |\varphi_2(\alpha\varepsilon)|\ge \house{\alpha\varepsilon}^{ \, \nu}.
 $$
We also denote by $\calEtilde_\nu^{(\alpha)}$ the set of units $\varepsilon\in\calE_\nu^{(\alpha)}$ such that $\varepsilon^{-1} \in\calE_\nu^{(1/\alpha)}$.\\
 
Let us state our main result.
\indent \begin{theoreme}\label{theoreme:principal}
Let $\nu \in 
]0,1[$. 
 There exist two effectively computable positive constants $ \Newcst{kappa1lambda}, \Newcst{kappa2lambda}$, depending
only upon $\alpha$ and $\nu$, which have the following properties:\\
\indent 
{\rm(a)} For any $m\ge 2$, each solution $(x,y,\varepsilon)\in\Z^2\times\calE_\nu^{(\alpha)}$ of the inequation $| F_\varepsilon(x,y)|\le m$ with $0<|x|\le |y|$ satisfies
$$
 \max\{|y|,\; e^{\rmh(\alpha\varepsilon)}\}\le m^{\cst{kappa1lambda}}.
$$ 
\indent {\rm(b)} For any $m\ge 2$, each solution $(x,y,\varepsilon)\in\Z^2\times\calEtilde_\nu^{(\alpha)}$ of the inequation $| F_\varepsilon(x,y)|\le m$ with $xy\neq 0$ satisfies
$$
 \max\{|x|,\; |y|,\; e^{\rmh(\alpha\varepsilon)}\}\le m^{\cst{kappa2lambda}}.
$$
 \end{theoreme}

Proposition $\ref{proposition:densitepositive}$, stated below and proved in \S\ref{S:DemPropositionDensite}, means that $\calEtilde_\nu^{(\alpha)}$
 for $d\ge 4$ has a positive density in the set $\calE ^{(\alpha)}$. Since the case of a non-totally real cubic field has been taken care of in $\cite{LW2}$, it is only in the case of a totally real cubic field that our main result
provides no effective bound for an infinite family of Thue equations.\\

When $N$ is a real positive number 
and $\calF$ is a subset of $\ZKtimes$, 
we define 
$$
 \calF(N)
 =\{\varepsilon\in \calF \; \mid \; \house{\alpha\varepsilon}\le N\}
 \quad
 \hbox{and}\quad
 |\calF(N)|=\Card \calF(N),
$$
so 
$$
\calF(N)
=
 \ZKtimes(N)\cap \calF.
$$

\begin{proposition}\label{proposition:densitepositive}
\null $\phantom{}$ 
{\rm (a)} 
The limit 
$$ 
\lim _{N\rightarrow \infty} \frac{ |\ZKtimes(N)|}{(\log N)^r}
$$ 
exists and is positive.
\\
\indent {\rm (b)} One has 
$$
\liminf _{N\rightarrow \infty} \frac{ |\calE^{(\alpha)}(N)|}{ |\ZKtimes(N)|}>0.
$$ 
\indent 
{\rm (c)} For $0<\nu<1/2$, one has 
$$
\liminf _{N\rightarrow \infty} \frac{ |\calE_\nu^{(\alpha)}(N)|}{(\log N)^r}>0.
$$ 
\indent {\rm (d)}
For $0<\nu<1$ and $d\ge 4$, one has 
$$
\liminf _{N\rightarrow \infty} \frac{ |\calEtilde_\nu^{(\alpha)}(N)|}{(\log N)^r}>0.
$$
\end{proposition}

 Let us write the irreductible polynomial $f$ of $\alpha$ over $\Z$ as
$$
f(X)=a_0X^d+a_1X^{d-1}+\cdots+a_{d-1}X+a_d\in\Z[X],
$$ 
whereupon 
$$
f(X)= a_0\prod_{i=1}^d \bigl(X-\sigma_i(\alpha)\bigr)
$$
and its associated irreducible binary form 
$F$ is 
$$
F(X,Y)=Y^d f(X/Y)=a_0X^d+a_1X^{d-1}Y+\cdots+a_{d-1}XY^{d-1}+a_dY^d.
$$
For $\varepsilon\in\ZKtimes$ verifying $\Q(\alpha\varepsilon)=K$, we have 
 $$
F_\varepsilon (X,Y)=a_0\prod_{i=1}^d \bigl(X-\sigma_i(\alpha\varepsilon)Y\bigr)\in\Z[X,Y].
$$
Given $(x,y,\varepsilon)\in\Z^2\times\ZKtimes$, we define 
$$
\beta=x- \alpha\varepsilon y.
$$
Therefore 
\begin{equation}\label{normebeta}
F_\varepsilon(x,y)=
a_0\sigma_1(\beta) \cdots \sigma_d(\beta).
\end{equation}

 Dirichlet's unit theorem provides the existence of units $\epsilon_1,\ldots,\epsilon_r$ in $K$, the classes modulo $K^\times_{\tors}$ of which form a basis of the 
free abelian group $\ZKtimes/K^\times_{\tors}$. Effective versions (see for instance $\cite{ST}$) provide bounds for the heights of these units as a function of $\rmh(\alpha)$ and $d$. \\

{\bf Steps of the proof.}
In \S$\ref{S:outils}$ we quote useful lemmas, the most powerful being a proposition of 
 $\cite{GL326}$ involving transcendence methods and giving lower bounds for the distance between 1 and a product of powers
 of algebraic numbers. Each time we will use that proposition, we will write that we are using a diophantine argument. 
After introducing some parameters $A$ and $B$ in \S$\ref{S:parametres}$,
we eliminate $x$ and $y$ between the equations $\varphi(\beta)=x-\varphi(\alpha\varepsilon)y$, $\varphi\in\Phi$. 
In \S$\ref{S:QuatreEnsemblesPlongements}$ we introduce four privileged embeddings, denoted by $\sigma_a$, $\sigma_b$, $\tau_a$, $\tau_b$, and four useful sets of embeddings $\Sigma_a(\nu)$, $\Sigma_b(\nu)$, $T_a(\nu)$, $T_b(\nu)$, depending on a parameter $\nu$.
Applying some results from $\cite{LW3}$, we show in 
 \S$\ref{S:minorationAetB}$ that we may suppose $A$ and $B$ sufficiently large, namely $\ge \kappa\log m$, via a diophantine argument. 
In \S$\ref{S:MajorationAparB}$ and in \S$\ref{S:DeuxiemeArgumentDiophantien}$, we prove that $A$ is bounded from above by $\kappa B$ and that 
 $B$ is bounded from above by $\kappa' \! A$. 
 In \S$\ref{S:UniciteTaub}$ we prove that $\tau_b$ is unique.
In \S$\ref{S:Evaluation}$ we give an upper bound for $|\tau_b(\alpha\varepsilon)|$. 
In \S$\ref{S:varphinotsigmaa}$ 
we deduce
 that $\sigma_a$ is unique.
In \S$\ref{S:DemThmPpal}$ we
 complete the proof of Theorem $\ref{theoreme:principal}$. In \S$\ref{S:DemPropositionDensite}$ we give the proof of Proposition $\ref{proposition:densitepositive}$.

\section{Tools}\label{S:outils}

This chapter contains the auxiliary lemmas we shall need. The details of the proofs are in $\cite{LW3}$. We start with an equivalence of norms (Lemma $\ref{Lemme:PlongementAdapte}$). Then we state Lemma $\ref{Lemme: LemmaA.15ST}$, which appeared as Lemma 
2 of $\cite{LW2}$ and also as Lemma 6 of $\cite{LW3}$. Next we quote Proposition $\ref{Proposition:FormeLineaireLogarithmes}$ (which is Corollary 9 of
 $\cite{LW3}$) involving a lower bound of a linear form in logarithms of algebraic numbers. 

\subsection{Equivalence of norms}\label{S:EquivalenceNormes}

Let $K$ be an algebraic number field of degree $d$ over $\Q$. Let us recall that $\epsilon_1,\ldots,\epsilon_r$ denote the 
elements of a basis of the unit group 
of $K$ modulo $K^\times_{\tors}$ and that we are supposing $r\ge 1$. \\

There exists an effectively computable positive constant $\Newcst{hauteurgamma}$, depending only upon $\epsilon_1,\ldots,\epsilon_r$, such that,
 if $c_1,\ldots, c_r$ are rational integers and if we let 
$$
C=\max\{ 
 |c_1|,\ldots,|c_r|\},\quad 
\gamma= \epsilon_1^{c_1}\cdots \epsilon_r^{c_r},
$$
 then 
\begin{equation}\label{Equation:EstimationsTriviales}
e^{-\cst{hauteurgamma} C }\le |\varphi(\gamma)|\le e^{ \cst{hauteurgamma} C }
\end{equation}
 for each embedding $\varphi$ of $K$ into $\C$. 
 
The following lemma (see Lemma 5 of \cite{LW3}) shows that the two inequalities of {\rm{($\ref{Equation:EstimationsTriviales}$)}} 
are optimal. 

\begin{lemme}\label{Lemme:PlongementAdapte}
There exists an effectively computable positive constant $\Newcst{kappa:PlongementAdapte}$, which depends only upon $\epsilon_1,\ldots,\epsilon_r$, with the following property. If $c_1,\ldots, c_r$ are rational integers and if we let 
$$
C=\max\{ 
 |c_1|,\ldots,|c_r|\},\quad 
\gamma= \epsilon_1^{c_1}\cdots \epsilon_r^{c_r},
$$
then there exist two embeddings $\sigma$ and $\tau$ of $K$ into $\C$ such that 
$$
|\sigma(\gamma) |\ge e^{\cst{kappa:PlongementAdapte}C }
\quad \hbox{and} \quad
|\tau(\gamma) |\leq e^{-\cst{kappa:PlongementAdapte} C}.
$$
\end{lemme}

\begin{remarque}
Under the hypotheses of Lemma $\ref{Lemme:PlongementAdapte}$, if $\gamma_0$ is a nonzero element of $K$ and if we let 
$\gamma_1=\gamma_0\gamma$, one deduces
$$
e^{-\cst{hauteurgamma} C -d\rmh(\gamma_0)}\le
\min_{\varphi\in\Phi} | \varphi(\gamma_1)|
\le
 e^{-\cst{kappa:PlongementAdapte} C +d \rmh(\gamma_0)}
$$
and
$$
e^{\cst{kappa:PlongementAdapte}C -d\rmh(\gamma_0) }
\le
\max_{\varphi\in\Phi} |\varphi(\gamma_1)|
\le
e^{ \cst{hauteurgamma} C +d\rmh(\gamma_0) }.
$$ 
 \end{remarque}

\subsection{On the norm}\label{SS:norme}

The following lemma is a consequence of Lemma A.15 of $\cite{ST}$ (see also Lemma 2 of $\cite{LW2}$ and Lemma 6 of \cite{LW3}). 

 \begin{lemme}\label{Lemme: LemmaA.15ST}
 Let $K$ be a field of algebraic numbers of degree $d$ over $\Q$ with regulator $R$. There exists an effectively computable positive constant $\Newcst{LemmaA.15STbis}$, depending only on $d$ and $R$, such that, if $\gamma$ is an element of $\ZK$, the norm of which has an absolute value $\le m$ with $m\ge 2$, then there exists a unit $\varepsilon\in\ZKtimes$ such that 
\begin{equation}\label{Equation:lemmeA.15ST}
 \max_{1\le j\le d} 
 \left| 
 \sigma_j(\varepsilon\gamma)
 \right|
 \le
m^{ \cst{LemmaA.15STbis}}.
 \end{equation} 
 \end{lemme}

\subsection{Diophantine tool}\label{SS:outilsdiophantiens}

We will use the particular case of Theorem 9.1 of $\cite{GL326}$ (stated in Corollary 9 of \cite{LW3}). Such estimates (known as {\it lower bounds for linear forms in logarithms of algebraic numbers}) first occurred in the work of A.O.~Gel'fond, then in the work of A.~Baker - a historical survey is given in \cite{LW3}. 
 
\begin{proposition}\label{Proposition:FormeLineaireLogarithmes}
Let $s$ and $D$ two positive integers. There exists an effectively computable positive constant $\Newcst{kappa:FLL}$, depending only upon $s$ and $D$, with the following property.
Let $\gamma_1,\ldots, \gamma_s$ be nonzero algebraic numbers generating a number field of degree $\le D$. Let $c_1,\ldots,c_s$ be rational integers and let $H_1,\ldots,H_s$ be real numbers $\ge 1$ satisfying $H_j\le H_s$ for $1\le j\le s$ and 
$$
H_i\ge \rmh(\gamma_i) \quad (1\le i\le s).
$$ 
 Let $C$ be a real number subject to 
$$
C\ge 2, \quad
C\ge
\max_{1\le j\le s} \left\{
 \frac{H_j}{H_s} |c_j|\right\}. 
$$
Suppose also $\gamma_1^{c_1}\cdots \gamma_s^{c_s} \not=1$.
Then 
$$
|\gamma_1^{c_1}\cdots \gamma_s^{c_s}-1|>
\exp\{-
\cst{kappa:FLL} H_1\cdots H_s\log C\}. 
$$
 \end{proposition}

\section{Introduction of the parameters $\Atilde$, $A$, $\Btilde$, $B$}\label{S:parametres}

From now on, we fix a solution $(x,y,\varepsilon)\in\Z^2\times\ZKtimes$ of the Thue inequation $| F_\varepsilon(x,y)|\le m$ with $xy\not=0$ 
and $\Q(\alpha\varepsilon)=K$. 
Up to \S\ref{S:varphinotsigmaa} inclusively, we suppose 
$$
1\le |x|\le |y|. 
$$
Let 
 $$
 \Atilde=\max\bigl\{1, \rmh(\alpha\varepsilon)\bigr\}. 
 $$
Write 
$$
\varepsilon=\zeta \epsilon_1^{a_1}\cdots\epsilon_r^{a_r}
$$
with $\zeta\in K^\times_{\tors}$ and $a_i\in\Z$ for $1\le i\le r$ and define 
$$
A=\max\{1,|a_1|, \dots,|a_r|\}.
$$
Thanks to {\rm{($\ref{Equation:EstimationsTriviales}$)}} and to Lemma $\ref{Lemme:PlongementAdapte}$, we have
$$
 \Newcst{minhauteuralphaepsilon} A\le \Atilde\le \Newcst{majhauteuralphaepsilon} A.
$$ 
\indent Next define
$$
\Btilde=\max\{1,\rmh(\beta)\}.
 $$
Since $|F_\varepsilon(x,y)|\le m$, it follows from {\rm{($\ref{Equation:lemmeA.15ST}$)}} and {\rm{($\ref{normebeta}$)}}
that there exists $\rho \in\ZK$ verifying 
\begin{equation}\label{equation:hauteurrho}
\rmh(\rho )\le \cst{ell} \log m
\end{equation}
with $\Newcst{ell}>0$ such that 
$\eta=\beta/\rho $ is 
a unit of $\ZK$ of the form 
 $$
\eta= \epsilon_1^{b_1}\cdots\epsilon_r^{b_r}
$$
with rational integers $b_1,\ldots,b_r$; define 
$$
B= \max\bigl\{1,\; |b_1|, |b_2|\ldots, \; |b_r|\bigr\}.
$$
Because of the relation $\beta=\rho \eta$, we deduce from {\rm{($\ref{Equation:EstimationsTriviales}$)}},
$$
 \Btilde\le \Newcst{majBtilde} (B +\log m)
$$
and from Lemma $\ref{Lemme:PlongementAdapte}$,
$$
B\le 
\Newcst{mminBtilde}(\Btilde +\log m).
$$
 
Since 
$xy\not=0$ and 
$ \Q(\alpha\varepsilon)=K$, 
we deduce that for $\varphi$ and $\sigma$ in $\Phi$, we have 
$$
\varphi=\sigma\;
\Longleftrightarrow \;
\varphi(\alpha\varepsilon)=\sigma(\alpha\varepsilon)\;
\Longleftrightarrow \;
\varphi(\beta)=\sigma(\beta)\;
\Longleftrightarrow \;
\sigma(\alpha\varepsilon)\varphi(\beta)=\sigma(\beta)\varphi(\alpha\varepsilon).
$$

 Here is an example of application of Proposition $\ref{Proposition:FormeLineaireLogarithmes}$.
The following lemma will be used 
in the proof of Lemma 
$\ref{Lemme:varphinotsigmaa}$. 

\begin{lemme}\label{Lemme:strategieA}
There exists an effectively computable positive constant $\Newcst{kappa:lemme:StrategieA}$ 
with the following property. 
Let $\varphi_1,\varphi_2,\varphi_3,\varphi_4$ be elements of $\Phi$ with $\varphi_1(\alpha\varepsilon)\varphi_2(\beta) \not=\varphi_3(\alpha\varepsilon)\varphi_4(\beta)$. Then 
$$
\left|
 \frac{\varphi_1(\alpha\varepsilon)\varphi_2(\beta) }
 {\varphi_3(\alpha\varepsilon)\varphi_4(\beta)} 
 -1\right|\ge 
 \exp\left\{
- \cst{kappa:lemme:StrategieA}
 (\log m)\log
 \left(2+
 \frac{A+B}{\log m}\right)
 \right\}.
$$
\end{lemme}

\begin{proof}[Proof]
Write 
$$
 \frac{\varphi_1(\alpha\varepsilon)\varphi_2(\beta) }
 {\varphi_3(\alpha\varepsilon)\varphi_4(\beta)} 
$$
as $\gamma_1^{c_1}\cdots \gamma_s^{c_s}$ with $s=2r+1$, and 
$$ 
 \gamma_j=\frac{ \varphi_1(\epsilon_j)}{\varphi_3(\epsilon_j)},
\quad
c_j=a_j,
\quad
 \gamma_{r+j}=\frac{ \varphi_2(\epsilon_j)}{\varphi_4(\epsilon_j)},
\quad
c_{r+j}=b_j
 \quad (j=1,\dots, r), 
$$
$$
\gamma_s=\frac{\varphi_1(\alpha\zeta)\varphi_2(\rho)} {\varphi_3(\alpha\zeta)\varphi_4(\rho)},
\quad c_{s}=1.
$$ 
We have $\rmh(\gamma_s)\le \Newcst{hauteurgammas} \log m$, thanks to the upper bound {\rm{($\ref{equation:hauteurrho}$)}} for the height of $\rho$. 
Write
$$ 
 H_1=\cdots=H_{2r}=\Newcst{HiLemmeStrategie}, \quad H_s=\cst{HiLemmeStrategie} \log m, \quad C=2+ \frac{A+B }{\log m}\cdotp
$$
The hypothesis 
$$
\max_{1\le j\le s} \frac {H_j}{ H_s} {|c_j| }\le C 
$$ 
of Proposition $\ref{Proposition:FormeLineaireLogarithmes}$ is satisfied. Lemma $\ref{Lemme:strategieA}$ follows from this proposition.

\end{proof}

\section{Elimination}\label{S:elimination}

\subsection{Expressions of $x$ and $y$ in terms of $\alpha\varepsilon$ and $\beta$}\label{SS:xety}
 
Let $\varphi_1,\varphi_2$ be two distinct elements of $\Phi$, namely two distinct embeddings of $K$ into $\C$. We eliminate $x$ (resp{.} $y$) between the two equations 
$$
\varphi_1(\beta)=x-\varphi_1(\alpha\varepsilon)y \quad \hbox{and}\quad 
\varphi_2(\beta)=x-\varphi_2(\alpha\varepsilon)y, 
$$
 to obtain
\begin{equation}\label{Equation:y}
y=\frac{\varphi_1(\beta)-\varphi_2(\beta)}{\varphi_2(\alpha\varepsilon)-\varphi_1(\alpha\varepsilon)},
\quad
x=\frac{\varphi_2(\alpha\varepsilon)\varphi_1(\beta)-\varphi_1(\alpha\varepsilon)\varphi_2(\beta)}
{\varphi_2(\alpha\varepsilon)-\varphi_1(\alpha\varepsilon)}
\cdotp
\end{equation}

\subsection{The unit equation 
}\label{S:Strategie}

Let $\varphi_1,\varphi_2,\varphi_3$ be embeddings of $K$ into $\C$. Let 
$$
u_i=\varphi_i(\alpha\varepsilon),\quad 
v_i=\varphi_i(\beta)
\qquad (i=1,2,3).
$$
We eliminate $x$ and $y$ between the three equations 
$$
\left\{
\begin{array}{lll}
\varphi_1(\beta)&=&x-\varphi_1(\alpha\varepsilon)y\\ [1mm] 
\varphi_2(\beta)&=&x-\varphi_2(\alpha\varepsilon)y\\ [1mm]
\varphi_3(\beta)&=&x-\varphi_3(\alpha\varepsilon)y\\
\end{array}\right.
$$
by writing that the determinant of this nonhomogeneous system of three equations in two unknowns, which is equal to 
$$
\left|
\begin{matrix}
1& \varphi_1(\alpha\varepsilon) &\varphi_1(\beta)\\ 
1& \varphi_2(\alpha\varepsilon) &\varphi_2(\beta)\\ 
1& \varphi_3(\alpha\varepsilon) &\varphi_3(\beta)\\
\end{matrix}
\right|=
\left|
\begin{matrix}
1&u_1& v_1\\ 
1&u_2& v_2\\
1&u_3& v_3\\
\end{matrix}
\right|,
$$
is 0, and this leads to 
\begin{equation}\label{Equation:SommeSixTermes}
u_1v_2-u_1v_3+u_2v_3-u_2v_1+u_3v_1-u_3v_2=0.
\end{equation}

\section{Four sets of privileged embeddings} \label{S:QuatreEnsemblesPlongements}

We denote by $\sigma_a$ (resp{.} $\sigma_b$) an embedding of $K$ into $\C$ such that $|\sigma_a(\alpha\varepsilon)|$ (resp{.} $|\sigma_b(\beta)|$) be
 maximal among the elements $|\varphi(\alpha\varepsilon)|$ (resp{.} among the elements $|\varphi(\beta)|$) for $\varphi\in\Phi$. Therefore 
$$
|\sigma_a(\alpha\varepsilon)|=\house{\alpha\varepsilon}
\quad
\hbox{and}
\quad
|\sigma_b(\beta)|=\house{\beta}\; .
$$ 
Next we denote 
by $\tau_a$ (resp{.} $\tau_b$) an embedding of $K$ into $\C$ such that $|\tau_a(\alpha\varepsilon)|$ (resp{.} $|\tau_b(\beta)|$) be minimal
among the elements $|\varphi(\alpha\varepsilon)|$ (resp{.} among the elements $|\varphi(\beta)|$) for $\varphi\in\Phi$. Therefore 
$$
\left|\tau_a\left((\alpha\varepsilon)^{-1}\right)\right|=
\frac{1}{\house{(\alpha\varepsilon)}}
\quad
\hbox{and}
\quad
\left|\tau_b\left(\beta^{-1}\right)\right|=
\frac{1}{\house{\beta}} \cdotp
$$ 
Since there are at least three distinct embeddings of $K$ into $\C$, we may suppose $\tau_b\not=\sigma_b $ and 
 $\tau_a \not=\sigma_a$. By definition of $\sigma_a$, $\sigma_b$, $\tau_a$ and $\tau_b$, for any $\varphi\in \Phi $ we have 
 $$
|\tau_a(\alpha\varepsilon)|\le |\varphi(\alpha\varepsilon)|\le |\sigma_a(\alpha\varepsilon)|
\quad
\hbox{and}
\quad
|\tau_b(\beta)|\le |\varphi(\beta)|\le |\sigma_b(\beta)|.
$$

Let $\thetanu$ be a real number in the open interval 
$]0,1[$. 
Let us denote by $\Sigma_a(\thetanu)$, \, $\Sigma_b(\thetanu), $\, $T_a(\thetanu), $ \, $ T_b(\thetanu)$\,\, 
the sets of embeddings of $K$ into $\C$ defined by the following conditions: 
$$
\left\{
\begin{array}{lll}
\Sigma_a(\thetanu)&=&\left\{
\varphi\in\Phi\; 
\mid
|\sigma_a(\alpha\varepsilon)|^\thetanu \le |\varphi(\alpha\varepsilon)|\le |\sigma_a(\alpha\varepsilon)|
\right\},
\\ [2mm]
\Sigma_b(\thetanu)&=&\left\{
\varphi\in\Phi\; 
\mid
|\sigma_b(\beta)|^\thetanu \le |\varphi(\beta)|\le |\sigma_b(\beta)|
\right\},
 \\ [2mm]
T_a(\thetanu)&=&\left\{
\varphi\in\Phi\; 
\mid
|\tau_a(\alpha\varepsilon)|\le |\varphi(\alpha\varepsilon)|\le |\tau_a(\alpha\varepsilon)|^\thetanu 
\right\},
\\ [2mm]
T_b(\thetanu)&=&\left\{
\varphi\in\Phi\; 
\mid
|\tau_b(\beta)|\le |\varphi(\beta)|\le |\tau_b(\beta)|^\thetanu 
\right\}.
\end{array} \right.
$$
Of course, we have 
$$
\sigma_a\in\Sigma_a(\thetanu),\quad
\sigma_b\in\Sigma_b(\thetanu),\quad 
\tau_a\in T_a(\thetanu),\quad
\tau_b\in T_b(\thetanu).
$$

We will see in \S$\ref{S:minorationAetB}$ that we have 
$$
|\sigma_a(\alpha\varepsilon)|>2,\quad 
|\sigma_b(\beta)|>2,\quad 
|\tau_a(\alpha\varepsilon)|<\frac{1}{2},\quad
|\tau_b(\beta)|<\frac{1}{2}, 
$$
from which we will deduce 
$$
 T_a(\thetanu)\cap\Sigma_a(\thetanu) =\emptyset,
 \quad
 T_b(\thetanu)\cap\Sigma_b(\thetanu) =\emptyset. 
 $$

\section{Lower bounds for $A$ and $B$}\label{S:minorationAetB}

Thanks to Lemma 15 in \S7.2 of $\cite{LW3}$ and to Lemma 17 in \S7.3 of 
$\cite{LW3}$, we may suppose, without loss of generality, that $A$ and $B$ have a lower bound given by $\Newcst{kappa:minorA} \log m$ for
a sufficiently large effectively computable positive constant $\cst{kappa:minorA}$, depending only on $\alpha$:
\begin{equation}\label{equation:minorationAetB}
A\ge \cst{kappa:minorA} \log m,\quad
B\ge \cst{kappa:minorA} \log m.
\end{equation}
In particular, we deduce that $A$, $B$, $|\sigma_a(\alpha\varepsilon)|$ and 
$|\sigma_b(\beta)| $ are sufficiently large and also that $|\tau_a(\alpha\varepsilon)|$ and $|\tau_b(\beta)|$ are sufficiently small.\\

 By using Lemma $\ref{Lemme:PlongementAdapte}$ with the estimates {\rm{($\ref{Equation:EstimationsTriviales}$)}}, we deduce
that there exist some effectively computable positive constants 
$\Newcst{kappa:majsigma}$ et $\Newcst{kappa:minsigma}$, depending only on $\alpha$, 
such that 
\begin{equation}
 \label{Equation:MajMinsigma}
\left\{
\begin{array}{lllll}
 e^{\cst{kappa:minsigma} A}&\leq& 
|\sigma_a(\alpha\varepsilon)| &\leq& e^{\cst{kappa:majsigma} A},\\ [1mm]
e^{\cst{kappa:minsigma} B}
&\leq& |\sigma_b(\beta)| &\leq& e^{\cst{kappa:majsigma} B}, \\ [1mm]
e^{-\cst{kappa:majsigma} A}& \leq& |\tau_a(\alpha\varepsilon)| &\leq &
e^{- \cst{kappa:minsigma} A},
\\ [1mm]
e^{-\cst{kappa:majsigma} B} 
 &\leq& |\tau_b(\beta)|&\leq& e^{- \cst{kappa:minsigma} B}.\\
\end{array}\right.
\end{equation}
 Therefore we have 
$$ \left\{
\begin{array}{llllllll} 
e^{\cst{kappa:minsigma}\thetanu A}& \leq& |\varphi(\alpha\varepsilon)|&\leq& e^{\cst{kappa:majsigma} A}
& \hbox{for} & \varphi\in 
\Sigma_a(\thetanu), 
\\ [1mm]
e^{\cst{kappa:minsigma}\thetanu B} &\leq& |\varphi(\beta)|&\leq& e^{\cst{kappa:majsigma} B}
& \hbox{for} &\varphi\in 
\Sigma_b(\thetanu), 
 \\ [1mm] 
e^{-\cst{kappa:majsigma} A}
&\leq& |\varphi(\alpha\varepsilon)|&\leq& 
e^{-\cst{kappa:minsigma}\thetanu B}
& \hbox{for} &\varphi\in 
T_a(\thetanu),
\\ [1mm] 
e^{-\cst{kappa:majsigma} B}
&\leq& |\varphi(\beta)|&\leq& 
e^{-\cst{kappa:minsigma}\thetanu B}
& \hbox{for}& \varphi\in 
T_b(\thetanu).
\end{array}\right.
$$

\section{Upper bounds for $A$, $|x|$, $|y|$ in terms 
 of $B$ 
}
\label{S:MajorationAparB}

From the relation {\rm{($\ref{Equation:y}$)}} we deduce in an elementary way the following upper bounds. Recall the assumption $1\le |x|\le |y|$ made in \S$\ref{S:parametres}$. 

\begin{lemme}\label{Lemme:BestGrand}
One has 
$$
A\le \Newcst{AmajoreparB} 
B
\quad\hbox{
and
}\quad
 |x| \le |y| \le e^{ \Newcst{yinferieurakappaB} B}.
$$
\end{lemme}
 
\begin{proof}[Proof] There is no restriction in supposing that $A$ and $B$ are larger than a constant times $\log m$.
From the inequality $|\sigma_a(\alpha\varepsilon)|\ge 2 |\tau_a(\alpha\varepsilon)| $, we deduce
$$
|\sigma_a(\alpha\varepsilon)-\tau_a(\alpha\varepsilon)|\ge
\frac{1}{2}|\sigma_a(\alpha\varepsilon)|.
$$
Then we use {\rm{($\ref{Equation:y}$)}} with $\varphi_2=\sigma_a$ and $\varphi_1=\tau_a$:
$$
y\bigl(\sigma_a(\alpha\varepsilon)-\tau_a(\alpha\varepsilon)\bigr)=
\tau_a(\beta) -\sigma_a(\beta).
$$
From the upper bound 
$$
|\sigma_a(\beta)-\tau_a(\beta)|\le 2|\sigma_b(\beta)|,
$$
we deduce
\begin{equation}\label{equation:majysigmaa}
|y\sigma_a(\alpha\varepsilon)|\le 4|\sigma_b(\beta)|.
\end{equation}
With the help of {\rm{($\ref{Equation:MajMinsigma}$)}}, one obtains the inequalities 
$$
e^{\cst{kappa:minsigma} A}\le |\sigma_a(\alpha\varepsilon) |\le | y\sigma_a(\alpha\varepsilon) | \le 4|\sigma_b(\beta)|
\le 4 e^{\cst{kappa:majsigma} B} 
$$
which imply $A\le \cst{AmajoreparB} B$.
From {\rm{($\ref{equation:majysigmaa}$)}} and because $|\sigma_a(\alpha\varepsilon) |> 2$, we get the upper bound $\log |y| \le \cst{yinferieurakappaB} B$.
We can conclude the proof by using the hypothesis $|x|\le |y|$
(cf{.} \S$\ref{S:parametres}$). 
\end{proof}

\section{Upper bound of $B$ in terms of $A$ 
}\label{S:DeuxiemeArgumentDiophantien}

We use the unit equation {\rm{($\ref{Equation:SommeSixTermes}$)}} of \S$\ref{S:Strategie}$ with three different embeddings $\tau_b$, $\sigma_b$ and $\varphi$, where $\varphi$
is an element of $\Phi$ different from $\tau_b$ and $\sigma_b$.

\begin{lemme}\label{Lemme:BmajoreParA}
One has 
$$
B \le \Newcst{kappa:BmajoreParA}A . 
$$
\end{lemme}

\begin{proof}[Proof]
Let $\varphi\in \Phi$ with $\varphi\not=\sigma_b$ and $\varphi\not=\tau_b$. We take advantage of the relation 
{\rm{($\ref{Equation:SommeSixTermes}$)}} with $\varphi_1=\sigma_b$, $\varphi_2=\varphi$, $\varphi_3 = \tau_b$,
 written in the form 
$$
 \varphi(\beta) 
 	\bigl(\sigma_b(\alpha\varepsilon)-\tau_b(\alpha\varepsilon)\bigr)
 -
\sigma_b(\beta) \bigl(\varphi(\alpha\varepsilon)-\tau_b(\alpha\varepsilon)\bigr)
+
\tau_b(\beta) \bigl(\varphi(\alpha\varepsilon)- \sigma_b(\alpha\varepsilon)\bigr)
=0
$$
and we divide by $\sigma_b(\beta) \bigl(\varphi(\alpha\varepsilon)- \tau_b(\alpha\varepsilon)\bigr)$ (which is different from 0): 
\begin{equation}\label{Equation:fllMajorationBparA}
\frac{ \varphi(\beta) }
{\sigma_b(\beta) }
 \cdot \frac{\sigma_b(\alpha\varepsilon)-\tau_b(\alpha\varepsilon)} { \varphi(\alpha\varepsilon)- \tau_b(\alpha\varepsilon)}
-1=
-
\frac{\tau_b(\beta)}
{\sigma_b(\beta) }
 \cdot \frac{ \varphi(\alpha\varepsilon)- \sigma_b(\alpha\varepsilon)} { \varphi(\alpha\varepsilon)- \tau_b(\alpha\varepsilon)}\cdotp
\end{equation}
The right side of {\rm{($\ref{Equation:fllMajorationBparA}$)}} is different from 0. Let us show that an upper bound of its modulus is given by 
$$
e^{\Newcst{ExpmoinskappaA}A} e^{-\Newcst{ExpmoinskappaB}B}.
$$
As a matter of fact, on the one hand, from {\rm{($\ref{Equation:MajMinsigma}$)}} we have 
$$
|\tau_b(\beta)| \le 
e^{- \cst{kappa:minsigma} B}, 
\quad \hbox{and}\quad |\sigma_b(\beta) |\ge e^{\cst{kappa:minsigma} B}; 
$$
on the other hand, the height of the number 
$$
\delta=
\frac{ \varphi(\alpha\varepsilon)- \sigma_b(\alpha\varepsilon)} { \varphi(\alpha\varepsilon)- \tau_b(\alpha\varepsilon)}
$$
 is bounded from above by $e^{\Newcst{kappa:majhauteurdeuxiemeargt} A}$. From this upper bound for the height we derive the upper bound for the modulus $|\delta|$, namely $|\delta|\le e^{\Newcst{kappa:minrdeuxiemeargt} A}$, hence
 $$
\left|
\frac{\tau_b(\beta)}
{\sigma_b(\beta) }
 \cdot \delta\right|
 \le 
 \frac{e^{\cst{kappa:minrdeuxiemeargt} A}}
{e^{2 \cst{kappa:minsigma} B}}\cdotp
 $$ 
Let us write the term 
$$
\frac{ \varphi(\beta) }
{\sigma_b(\beta) }
 \cdot \frac{\sigma_b(\alpha\varepsilon)-\tau_b(\alpha\varepsilon)} { \varphi(\alpha\varepsilon)- \tau_b(\alpha\varepsilon)}
$$
 appearing on the left side of {\rm{($\ref{Equation:fllMajorationBparA}$)}} in the form $\gamma_1^{c_1}\cdots \gamma_s^{c_s}$ with $s=r+1$ and 
$$ 
 \gamma_j=\frac{ \varphi(\epsilon_j)}{\sigma_b(\epsilon_j)},
\quad
c_j=b_j \quad (j=1,\dots, r), 
$$
$$
\gamma_s=\frac{\sigma_b(\alpha\varepsilon)-\tau_b(\alpha\varepsilon)} { \varphi(\alpha\varepsilon)- \tau_b(\alpha\varepsilon)}
\cdot
\frac{\varphi(\rho )}{\sigma_b(\rho )},
\quad c_s=1.
$$ 
Thanks to {\rm{($\ref{equation:hauteurrho}$)}} and {\rm{($\ref{equation:minorationAetB}$)}}, we have 
$$
\rmh(\gamma_s)\le
 \Newcst{hauteurgamma0} A +2\rmh(\varrho)\le
 \Newcst{hauteurgamma0bis} A.
$$
Define 
$$ 
 H_1=\cdots=H_{r}=\Newcst{MinorationAparB}, \quad H_s=\cst{hauteurgamma0} A, \quad C=2+\frac{B }{\Newcst{kappa:majAparB}A}\cdotp
$$
We check that the hypothesis 
$$
\max_{1\le j\le s} \frac {H_j}{ H_s} {|c_j| }\le C
$$ 
of Proposition $\ref{Proposition:FormeLineaireLogarithmes}$ is satisfied. We deduce from this proposition that a lower bound for the modulus of the left member of {\rm{($\ref{Equation:fllMajorationBparA}$)}} is given by $\exp\{-\Newcst{AlogB} H_s\log C \}$. Consequently,
 $$
 \cst{ExpmoinskappaB}B \le \cst{ExpmoinskappaA}A+ \cst{AlogB} H_s\log C .
 $$
 Hence $C\le \Newcst{MajfinaleB}\log C $, which allows to conclude that $C\le \Newcst{kappa:Cmajore}$, and this secures the inequality 
 $B\le \cst{kappa:BmajoreParA}A$ we wanted to prove. 
 \end{proof}

\section{ Unicity of $\tau_b$ 
}\label{S:UniciteTaub}

We want to prove that no other embedding plays the same role as $\tau_b$. This will be achieved by proving the next lemma, which exhibits a contradiction to {\rm{($\ref{equation:minorationAetB}$)}}. 

\begin{lemme}\label{Lemme:UniciteTau} 
Suppose $T_b(\thetanu)\neq \{\tau_b\}$. 
 Then $ B \le \Newcst{UniciteTau} \log m$. 
\end{lemme}

\begin{proof}[Proof]
Let $\varphi\in T_b(\thetanu)$. Suppose $\varphi\not=\tau_b$.
Let us use 
 {\rm{($\ref{Equation:y}$)}} with $\varphi_1=\varphi$, $\varphi_2=\tau_b$, in the form 
 $$
 \frac{\varphi(\alpha\varepsilon)}{\tau_b(\alpha\varepsilon)}-1
 =\frac{\tau_b(\beta)- \varphi(\beta)}{y \tau_b(\alpha\varepsilon)}\cdotp
$$ 
From the inequality
$$
|x-\tau_b(\alpha\varepsilon) y| = 
|\tau_b(\beta)| 
 <\frac{1}{2}
$$
obtained from {\rm{($\ref{Equation:MajMinsigma}$)}}, 
we deduce 
$$
|\tau_b(\alpha\varepsilon) y|\ge |x|-\frac{1}{2}\ge \frac{1}{2}\cdotp
$$
Since $|\tau_b(\beta)|\le |\varphi(\beta)|$, we also have
$$
|\varphi(\alpha\varepsilon) -\tau_b(\alpha\varepsilon) | 
=
\frac{1 }{y} |\varphi(\beta) -\tau_b(\beta)
|\le 
\frac{2|\varphi(\beta)|}{|y|}\cdotp
$$
Consequently, 
$$
\left|
\frac{\varphi(\alpha\varepsilon) }{\tau_b(\alpha\varepsilon) }-1
\right|
\le 
\frac{
2 |\varphi(\beta) |}
{|\tau_b(\alpha\varepsilon) y|} 
\le 4 |\varphi(\beta) |
 \le 4 e^{- \cst{kappa:minsigma} \thetanu B}. 
$$ 
The left side is not 0 since $\varphi\not=\tau_b$. Let us write 
$$
\frac{\varphi(\alpha\varepsilon) }{\tau_b(\alpha\varepsilon) }=\gamma_1^{c_1}\cdots \gamma_s^{c_s}
$$
with 
$s=r+1$, and 
$$
\gamma_i=\frac{\varphi(\epsilon_i)}{ \tau_b(\epsilon_i)}, 
\quad
c_i=a_i, 
\quad (i=1,\dots,r), 
\quad 
\gamma_s= \frac{\varphi(\alpha\zeta)}{\tau_b(\alpha\zeta)},
\quad
c_s=1 . 
$$
From Proposition $\ref{Proposition:FormeLineaireLogarithmes}$ with 
$$
H_1=\cdots=H_s=\Newcst{Hi},\quad C=A,
$$
we deduce $B\le \Newcst{kappa:BmajoreParAbis} \log A $. Then we use the upper bound $A\le \cst{AmajoreparB} 
B $ of Lemma $\ref{Lemme:BestGrand}$ to get $B\le \Newcst{kappa:BmajoreParLogB} \log B $ and 
$A\le \Newcst{kappa:AmajoreParLogA} \log A $. We use 
 {\rm{($\ref{Equation:MajMinsigma}$)}}
to conclude the proof of Lemma $\ref{Lemme:UniciteTau}$.
\end{proof}

Therefore Lemma 
$\ref{Lemme:UniciteTau}$ now allows us to suppose that 
for any $\varphi\in \Phi$ different from $\tau_b$, 
we have $|\varphi(\beta)|>|\tau_b(\beta)|^{\thetanu}$.
In particular, the embedding $\tau_b$ is then real. This is the end of the proof in the totally imaginary case, cf. $\cite{LW3}$.\\

From now on, we suppose $T_b(\thetanu) = \{ \tau_b\}$.

\section{ Upper bound for $|\tau_b(\alpha\varepsilon)|$ 
}\label{S:Evaluation}

An upper bound for $|\tau_b(\alpha\varepsilon)|$
is exhibited.

\begin{lemme}\label{Lemme:taubalphaetsigmaabeta}
One has 
$
|\tau_b(\alpha\varepsilon)| \le 2.
$
 
\end{lemme}

\begin{proof}[Proof] 
We have 
$$
|x-\tau_b(\alpha\varepsilon) y| =|\tau_b(\beta)| < \frac{1}{2} < |x|,
$$
wherupon we deduce 
$$ 
|\tau_b(\alpha\varepsilon) y|\le 2|x|\le 2|y|,
$$
since $|x|\le |y|$. 
 
\end{proof}

\section{Unicity of $\sigma_a$ 
}
\label{S:varphinotsigmaa}

Since $|\tau_b(\beta)|$ is very small, $x$ is close to $\tau_b(\alpha\varepsilon)y$. Now, for any $\varphi\in\Phi$, we have 
 $$\varphi(\beta)=x-\varphi(\alpha\varepsilon) y.
$$
Consequently, if $|\varphi(\alpha\varepsilon)|$ is smaller than $|\tau_b(\alpha\varepsilon)|$, then $\varphi(\beta)$ is close to $x$, while 
if $|\varphi(\alpha\varepsilon)|$ is larger than $|\tau_b(\alpha\varepsilon)|$, then $\varphi(\beta)$ is close to $-\varphi(\alpha\varepsilon)y$. 
Let us justifty these claims. 

\begin{lemme}\label{Lemme:varphibeta}
Let $\varphi\in\Phi$. \\
\indent 
 {\rm(a)} Let $\lambda$ be a real number in the interval 
 $]0,1[ $. 
If $|\varphi(\alpha\varepsilon)|\le \lambda |\tau_b(\alpha\varepsilon)|$, then 
$$
|\varphi(\beta) -x|\le \lambda|x| + \lambda e^{- \cst{kappa:minsigma} B}.
$$ 
\indent {\rm(b)} Let $\mu$ be a real number $>1$. 
 If $|\varphi(\alpha\varepsilon)|\ge \mu |\tau_b(\alpha\varepsilon)|$, then 
$$
|\varphi(\beta) +\varphi(\alpha\varepsilon)y|\le \frac{1}{\mu} |\varphi(\alpha\varepsilon)y| + e^{- \cst{kappa:minsigma} B}.
$$
\end{lemme}

\begin{proof}[Proof]
We have 
$ |\tau_b(\beta)|\le e^{- \cst{kappa:minsigma} B}$, 
namely
$$
|x-\tau_b(\alpha\varepsilon)y|\le e^{- \cst{kappa:minsigma} B}.
$$
We also have 
$$
\varphi(\beta)=x-\varphi(\alpha\varepsilon)y.
$$
Because of the hypothesis {\rm(a)}
we get
$$
|\varphi(\beta) -x|
=|\varphi(\alpha\varepsilon)y |\le \lambda | \tau_b(\alpha\varepsilon)y|\le \lambda |x|+ \lambda e^{- \cst{kappa:minsigma} B}.
$$
Because of the hypothesis (b), we have 
$$
|\varphi(\beta) +\varphi(\alpha\varepsilon)y|=|x|\le | \tau_b(\alpha\varepsilon)y| + e^{-\ \cst{kappa:minsigma} B}\le 
 \frac{1}{\mu} |\varphi(\alpha\varepsilon)y| + e^{-\ \cst{kappa:minsigma} B}.
 $$
 \end{proof}

\begin{lemme}\label{Lemme:varphinotsigmaa}
Let $\varphi\in\Phi$ with $\varphi\not=\sigma_a$. Then 
$$
|\varphi(\beta)|\le 2 |x|
\exp\left\{\Newcst{kappa:varphibeta} (\log m)
\log \left(2+\frac{A+B}{\log m}\right)\right\}
$$
and 
$$
|\varphi(\alpha\varepsilon)|\le
\max\left\{
\frac{3}{2} |\tau_b(\alpha\varepsilon)|\; , \; 
 8\frac{|x|}{|y|} 
 \exp\left\{ \cst{kappa:varphibeta} (\log m)
 \log \left(2+\frac{A+B}{\log m}\right)\right\}
 \right\}.
$$
\end{lemme}

\begin{proof}[Proof]
From the relation {\rm{($\ref{Equation:y}$)}} with $\varphi_1=\sigma_a$ et $\varphi_2=\varphi$ we deduce
$$
x=\frac{\varphi(\alpha\varepsilon)\sigma_a(\beta)-\sigma_a(\alpha\varepsilon)\varphi(\beta)}
{\varphi(\alpha\varepsilon)-\sigma_a(\alpha\varepsilon)},
$$
hence 
$$
\frac{\varphi(\alpha\varepsilon)\sigma_a(\beta)}{\sigma_a(\alpha\varepsilon)\varphi(\beta)}-1
=
\frac{\varphi(\alpha\varepsilon)-\sigma_a(\alpha\varepsilon)}{\sigma_a(\alpha\varepsilon)\varphi(\beta)}\cdot x.
$$
The member of the right side is nonzero, and its modulus is bounded from above by $2|x|/|\varphi(\beta)|$ since 
 $|\varphi(\alpha\varepsilon)|\le |\sigma_a(\alpha\varepsilon)|$. 
The upper bound of $|\varphi(\beta)|$ follows from Lemma $\ref{Lemme:strategieA}$ with 
 $\varphi_1=\varphi_4=\varphi$ et $\varphi_2=\varphi_3=\sigma_a$.\\

To establish the upper bound given in Lemma $\ref{Lemme:varphinotsigmaa}$ for $|\varphi(\alpha\varepsilon)|$, we may suppose
$$
|\varphi(\alpha\varepsilon)| > \frac{3}{2} |\tau_b(\alpha\varepsilon)|,
$$
otherwise the conclusion is trivial. Then we may use Lemma 
 $\ref{Lemme:varphibeta}$ 
(b) with $\mu=3/2$ to deduce
$$
|\varphi(\alpha\varepsilon)y|-|\varphi(\beta)|\le
|\varphi(\beta)+\varphi(\alpha\varepsilon)y|\le \frac{2}{3} |\varphi(\alpha\varepsilon)y|+ e^{-\ \cst{kappa:minsigma} B},
$$
hence 
$$
|\varphi(\alpha\varepsilon)y|\le 3 |\varphi(\beta)| +3e^{-\ \cst{kappa:minsigma} B}\le 4 \max \{|\varphi(\beta)|, 1 \}.
$$
We can conclude by using the upper bound of $|\varphi(\beta)|$ which we just established.
\end{proof}

 From Lemma $\ref{Lemme:varphinotsigmaa}$ we deduce the following.
 
\begin{corollaire}\label{Corollaire:varphinotsigmaa}
Assuming {\rm{($\ref{equation:minorationAetB}$)}}, we have $\Sigma_a(\thetanu)=\{\sigma_a\}$.
\end{corollaire}


 
 \begin{proof}[Proof]
 Let us remind that $|x|\le |y|$. 
Since $|\tau_b(\alpha\varepsilon)|\le 2$ (Lemma $\ref{Lemme:taubalphaetsigmaabeta}$), with $\sigma\in \Sigma_a(\thetanu)$, we have 
$$
|\sigma(\alpha\varepsilon)|>
\frac{3}{2} |\tau_b(\alpha\varepsilon)|.
$$
If there were $\sigma\in\Sigma_a(\thetanu)$ with $\sigma\not=\sigma_a$, by using Lemma $\ref{Lemme:varphinotsigmaa}$ with $\varphi=\sigma$, we would deduce $A\le \Newcst{kappa:sigmaaunique} \log m$ and thanks to {\rm{($\ref{equation:minorationAetB}$)}} we could conclude that $\sigma_a$ is the only element of 
$\Sigma_a(\thetanu)$.
 \end{proof}

\section{Proof of of the main result}
\label{S:DemThmPpal}

Let us concentrate on the 

\begin{proof}[Proof
 of Theorem $\ref{theoreme:principal}$]


For the part {\rm(a)} of Theorem $\ref{theoreme:principal}$, we take $\varepsilon\in \calE_\nu^{(\alpha)}$; by definition of $\calE_\nu^{(\alpha)}$, there exists $\varphi\in\Phi$, $\varphi\not=\sigma_a$, with 
$$
|\varphi (\alpha\varepsilon)|\ge |\sigma_a (\alpha\varepsilon)|^\nu,
$$
namely $\varphi\in\Sigma_a(\thetanu)$. Since 
$\Sigma_a(\thetanu)$ contains more than one element, 
Corollary $\ref{Corollaire:varphinotsigmaa}$ shows that the inequalities {\rm{($\ref{equation:minorationAetB}$)}} are not satisfied. This completes the proof of part {\rm(a)} of Theorem $\ref{theoreme:principal}$.\\

To prove the part {\rm(b)}, we will use 
the reciprocal polynomial of $f_{\varepsilon}$, defined by 
 $$
 Y^df_{\varepsilon}(1/Y)=a_dY^d+\cdots+a_0=a_d \prod_{i=1}^d \bigl(Y-\sigma_i(\alpha'\varepsilon')\bigr),
 $$
 with $\alpha'=\alpha^{-1}$ and $\varepsilon'=\varepsilon^{-1}$
and we will write the binary form $F_\varepsilon$ as 
 $$
F_\varepsilon (X,Y)=a_d\prod_{i=1}^d \bigl(Y-\sigma_i(\alpha'\varepsilon')X\bigr). 
$$
The part {\rm(a)} of Theorem $\ref{theoreme:principal}$ not only indicates that any solution $(x,y,\varepsilon)\in\Z^2\times\calE_\nu^{(\alpha)}$ of
the inequation $| F_\varepsilon(x,y)|\le m$ with $0<|x|\le |y|$ verifies
$$
 \max\{|y|,\; e^{\rmh(\alpha\varepsilon)}\}\le m^{\cst{kappa1lambda}(\alpha)},
$$ \\ 	
but also shows that 
any solution $(x,y,\varepsilon)\in\Z^2\times\calE_\nu^{(\alpha)}$ of the inequation $| F_\varepsilon(x,y)|\le m$ with $0<|y|\le |x|$ verifies
$$
 \max\{|x|, \; e^{\rmh(\alpha'\varepsilon')}\}\le m^{\cst{kappa1lambda}(\alpha')}.
$$ 	
Since $\rmh(\alpha'\varepsilon')=\rmh(\alpha\varepsilon)$ and since $\calEtilde_\nu^{(\alpha)}$ is the set of $\varepsilon\in\calE_\nu^{(\alpha)}$ such that $\varepsilon'\in\calE_\nu^{(\alpha')}$, it follows that each solution of the inequation $| F_\varepsilon(x,y)|\le m$ with $xy\neq 0$ verifies
$$
 \max\{|x|,\; |y|,\; e^{\rmh(\alpha\varepsilon)}\}\le m^{\cst{kappa2lambda}}.
$$
\end{proof} 

\section{Proof of Proposition $\ref{proposition:densitepositive}$}
\label{S:DemPropositionDensite}

Let us index the elements of $\Phi$ in such a way that $\sigma_1,\dots,\sigma_{r_1}$ are the real embeddings and $ \sigma_{r_1+1},\dots, \sigma_d$ are the non-real embeddings, with $\sigma_{r_1+j}=\overline{\sigma}_{r_1+r_2+j}$ ($1\le j\le r_2$). We have $d=r_1+2r_2$ and $r=r_1+r_2-1$. 
The logarithmic embedding of $K$ is the group homomorphism $\ulambda$ of $K^\times$ into $\R^{r+1}$ defined by 
$$
\ulambda(\gamma)= 
\bigr(
\delta_1 \log |\sigma_1(\gamma)|, \dots, \delta_{r+1}\log |\sigma_{r+1}(\gamma)|\bigr), 
$$
where
$$
\delta_i=
\begin{cases}
1& \hbox{for $i=1,\dots,r_1$},
\\
2& \hbox{for $i=r_1+1,\dots,r_1+r_2$}.
\end{cases}
$$ 
Its kernel is the finite subgroup $K^\times_{\tors}$ of torsion elements of $K^\times$, which are the roots of unity belonging to $K$. 
By Dirichlet's theorem, the image of $\ZKtimes$ under $\ulambda$ is a lattice of the hyperplane $\cal H$ of equation 
\begin{equation}\label{Equation:hyperplanDirichlet}
t_1+\cdots+t_{r+1}=0
\end{equation} 
in $\R^{r+1}$.
For $\borne >0$, define 
 $$
\calH (M) = \{(t_1,\dots,t_{r+1}) \in \calH \; \mid \; \max \{\delta_1^{-1}t_1, \dots, \delta_{r+1} ^{-1} t_{r+1} \} \leq \borne
\}.
$$ 
For all elements $(t_1,\dots,t_{r+1})$ of $\calH (\borne)$ we have 
$$
\max_{1\leq i \leq r+1} t_i \leq 2 \borne.
$$
Further, the inequality
$$
t_1+\cdots+t_{r+1}\le \min_{1\leq i \leq r+1} t_i+r \max_{1\leq i \leq r+1} t_i
$$
together with the equation of $\calH$ implies
$$
\max_{1\leq i \leq r+1} -t_i = - \min_{1\leq i \leq r+1} t_i
\leq 
r \max_{1\leq i \leq r+1} t_i \leq 2r \borne,
$$ 
hence this set $ \calH (M) $ is bounded: namely, for $ (t_1,\dots,t_{r+1}) \in \calH (M)$, 
$$ 
 \max_{1\leq i \leq r+1} |t_i| \leq 2r \borne.
$$
The $r$--dimension volume of 
$\calH (\borne)$
is the product of the volume of $ \calH (1)$ by $\borne^r$ 
 while the volume of $\calH (1)$ is an effectively computable positive constant, depending only upon $r_1$ and $r_2$. \\

{\it Proof of the part {\rm(a)}}.
Since $\ulambda(\ZKtimes)$ is a lattice of the hyperplane 
 $\cal H$, the limit
$$
\lim_{\borne\rightarrow\infty} \frac{1}{\borne^r} \left| \ulambda(\ZKtimes) \cap \calH(\borne) \right|
$$
exists and is a positive number. \\

The image of $\varepsilon \in \ZKtimes$ by $\ulambda$ is
$$
\ulambda( \varepsilon)=(t_1,\dots,t_{r+1})
\quad\hbox{
with 
}\quad
t_i=\delta_i\log |\sigma_i(\varepsilon)| \quad (i=1,\dots,r+1).
$$
If on the one hand $\varepsilon \in \ZKtimes(N)$, then 
$$
\delta_i^{-1}t_i
\le \log N -\log |\sigma_i(\alpha)| \quad (1\le i\le r+1);
$$
therefore
$$
\max_{1\le i\le r+1} \delta_i^{-1}t_i \le \log N + \log \house{\alpha^{-1}}\, .
$$
Consequently, if we define 
$$ 
\borne_+=\log N + \log \house{\alpha^{-1}} \ ,
$$ 
we have $\ulambda( \varepsilon)\in \calH(\borne_+)$. 
On the other hand, we have 
$$
\log | \sigma_i(\alpha\varepsilon)| =\delta_i^{-1}t_i + \log | \sigma_i(\alpha)| \le \delta_i^{-1}t_i+ \log \house{\alpha} \, .
$$
If we define 
$$
\borne_-= \log N- \log \house{\alpha}\, ,
$$
then for any $\ulambda( \varepsilon)\in \ulambda\bigl(\ZKtimes \bigr) \cap \calH(\borne_-) $ we have $\varepsilon\in \ZKtimes(N)$. Therefore,
$$
 \ulambda\bigl(\ZKtimes \bigr) \cap \calH(\borne_-) \subset \ulambda\bigl(\ZKtimes(N)\bigr)\subset
 \ulambda\bigl(\ZKtimes \bigr) \cap \calH(\borne_+).
$$
Now we can conclude that the part {\rm(a)} of 
 Proposition $\ref{proposition:densitepositive}$ 
is proved.\\

Recall that a CM field is a totally imaginary number field which is a quadratic extension of its maximal totally real subfield. 
Let us prove that for a CM field the number of elements $\varepsilon$ of $\ZKtimes(N)$ such that $\Q(\alpha\varepsilon)\not=K$ is negligible with respect to 
the number of elements $\varepsilon$ of $\ZKtimes(N)$ such that $\Q(\alpha\varepsilon)=K$. 
 Denote by $\calF^{(\alpha)}$ the complement of $\calE^{(\alpha)}$ in $\ZKtimes$: 
$$
\calF^{(\alpha)}=\{\varepsilon\in \ZKtimes\; \mid \; \Q(\alpha\varepsilon)\not=K\}.
$$
 
\begin{lemme}\label{lemme:Qalphaepsilon=K}
Assume $K$ is not a CM field. Then 
$$
\limsup_{N\rightarrow \infty} 
\frac{1}{(\log N)^{r-1}} 
\left|\ulambda
\bigl(
\calF^{(\alpha)}(N) \bigr)
\right|
<\infty.
$$
\end{lemme}

\begin{proof}[Proof] 
The set of subfields $L$ of $K$ is finite. Since $K$ is not a CM field, the rank $\varrho$ of the unit group of such a subfield $L$ strictly contained in $K$ is smaller than $r$. Therefore 
 the number of $\varepsilon\in\ZKtimes$ such that $\Q(\alpha\varepsilon)=L$ and $\lambda(\alpha)\in \calH(\borne)$ is bounded by a constant times $\borne^\varrho$. The proof of Lemma $\ref{lemme:Qalphaepsilon=K}$ is then secured.
 
\end{proof}

{\it Proof of the part {\rm(b)} of Proposition $\ref{proposition:densitepositive}$}. 
If $K$ is not a CM field, the stronger estimate 
$$
\lim _{N\rightarrow \infty} \frac{ |\calE^{(\alpha)}(N)|}{ |\ZKtimes(N)|}=1.
$$ 
 follows from the part 
{\rm(a)} and from Lemma $\ref{lemme:Qalphaepsilon=K}$.

Assume now that $K$ is CM field with maximal totally real subfield $L_0$. Since $K=\Q(\alpha)$, for any $\varepsilon\in \Z_{L_0}^\times$, we have $\Q(\alpha\varepsilon)\not=L_0$. As we have seen, for each subfield $L$ of $K$ different from $K$ and from $L_0$, the set of $\varepsilon\in\ZKtimes$ such that $\Q(\alpha\varepsilon)=L$ and $\lambda(\alpha)\in \calH(\borne)$ is bounded by a constant times $\borne^{r-1}$. The other elements $\varepsilon\in\ZKtimes$ with $\lambda(\alpha)\in \calH(\borne)$ have $\Q(\alpha\varepsilon)=K$. This completes the proof of the part {\rm(b)} of Proposition $\ref{proposition:densitepositive}$.

Before completing the proof of Proposition $\ref{proposition:densitepositive}$, one introduces a change of variables $t_i=\delta_ix_i$: we call $H$ the hyperplane of $\R^{r+1}$ of equation 
 $$
 \delta_1x_1+\cdots+\delta_{r+1}x_{r+1}=0
 $$
 and for $\borne>0$, we consider
 $$
 H(\borne)
 = \big\{
(x_1,\dots,x_{r+1})\in \calH \; \mid \; \max \{x_1,\dots,x_{r+1}\}\le \borne
\big \}.
 $$
 
{\it Proof of the part {\rm(c)}}. 

Let $\nu$ be a real number in the interval 
$]0,1[$. 
Let us take $\borne=\log N$.
Define some subsets $D_\nu(\borne)$ and $ D'_\nu(\borne)$ of $H(\borne)$ the following way:
\begin{align}\notag
&D_\nu(\borne) = \bigl\{
(x_1,\dots,x_{r+1})\in H(\borne) \; \mid \; 
\\
\notag&
\hbox{$\qquad$ 
there exists $i,j$ with $ i \not= j$ and $1\le i, j\le r_1$
 such that $x_i \ge \nu \borne$ and $x_j \ge \nu \borne$}
\bigr\}
\end{align}
and 
\begin{align}\notag
& D'_\nu(\borne) = \bigl\{
(x_1,\dots,x_{r+1})\in H(\borne) \; \mid \; 
\\&
\notag
\hbox{$\qquad$ there exists $i $ with $r_1<i\le r+1$, such that $x_i \ge \nu \borne\}
 \bigr\}. \qquad\qquad\qquad$}
\end{align}
If $D_\nu(\borne)$ is not empty, then $r_1\ge 2$ while if $D'_\nu(\borne)$ is not empty, then $r_2\ge 1$. 
We show that if $r_1\ge 2$ and $0<\nu<\delta_{r+1}/2$, then $D_\nu(1)$ has a positive volume while if $r_2\ge 1$ and $0<\nu<\delta_r/2$, then $ D'_\nu(1)$ has a positive volume.
This will show that, for a number field of degree $\ge 3$ and for $0<\nu<1/2$, at least one of the two sets $D_\nu(1)$ and $ D'_\nu(1)$ has a positive volume.\\

Assume $r_1\ge 2$, hence $\delta_1=\delta_2=1$, and $0<\nu<\delta_{r+1}/2$. Let $a$, $b$, $c$ be positive real numbers with $\nu\le a<b<\delta_{r+1}/2$, $c<1$ and $c<\delta_{r+1}-2b$. Then $D_\nu(1) $ contains the set of $(x_1,x_2,\ldots,x_{r+1})\in H$ verifying
\footnote{Notice that one does not divide by $0$: if $r=2$ the last conditions for $3\le i\le r$ disappear.}
$$
a\le x_1, x_2 \le b, \quad \frac{-c}{\delta_i(r-2)}\le x_i\le \frac{c}{\delta_i(r-2)} \quad (3\le i\le r),
$$
because these bounds and the equation $\delta_1x_1+\delta_2x_2+\cdots+\delta_{r+1}x_{r+1}=0$ of $H$, imply
$$
-1\le x_{r+1} \le 1. 
$$ 
This shows that $D_\nu(1) $ has positive volume.

Next assume $r_2\ge 1$, hence $\delta_{r+1}=2$, and $0<\nu<\delta_r/2$. Let $a$, $b$, $c$ be positive real numbers with $\nu\le a<b<\delta_r/2$ and $c<\delta_r-2b$. Then $D'_\nu(1) $ contains the set of $(x_1,x_2,\ldots,x_{r+1})\in H$ verifying
$$
a\le x_{r+1} \le b, \quad \frac{-c}{\delta_i(r-1)}\le x_i\le \frac{c}{\delta_i(r-1)} \quad (1\le i\le r-1),
$$
because these bounds, together with the equation $\delta_1x_1+\delta_2x_2+\cdots+\delta_{r+1}x_{r+1}=0$ of $H$, imply
$$
-1\le x_r \le 1. 
$$ 
This shows that $D'_\nu(1) $ has positive volume.
 
 Once we know that the 
 $r$--dimension volume of $D_\nu(1)$ (resp. $ D'_\nu(1)$) in $H$ is positive, we deduce that the $r$--dimension volume of $D_\nu(\borne)$ (resp. $ D'_\nu(\borne)$)
is bounded below by an effectively computable positive constant times $\borne^r$ --- as a matter of fact, $D_\nu(\borne)$ (resp. $ D'_\nu(\borne)$) is equal to the product of $\borne^r$ by the effectively computable constant $D_\nu(1)$ (resp. $ D'_\nu(1)$). Since $\ulambda(\alpha)+\ulambda(\ZKtimes)$ is a translate of the lattice $\ulambda(\ZKtimes)$, 
the cardinality of the set 
$$
\bigl(\ulambda(\alpha)+\ulambda(\ZKtimes)\bigr)\cap\big( D_\nu(\borne)\cup D'_\nu(\borne)\big)
$$
 is bounded below by an effectively computable positive constant times $\borne^r$.\\
 
Let $\varepsilon\in\ZKtimes$ be such that $\ulambda(\alpha\varepsilon)\in D_\nu(\borne)\cup D'_\nu(\borne)$. We have 
$$
\log \max_{1\le j\le d} |\sigma_j(\alpha\varepsilon)|\le \borne 
 $$
and there exist two distinct elements $\varphi_1,
\varphi_2$ of $\Phi$ such that 
 $$
\log|\varphi_i(\alpha\varepsilon)|\ge \nu \borne \quad (i=1,2).
$$
Consequently, 
$$
\house{\alpha\varepsilon}\le e^\borne,\quad 
|\varphi_1(\alpha\varepsilon)|\ge \house{\alpha\varepsilon}^{\, \nu}, 
\quad|\varphi_2(\alpha\varepsilon)|\ge \house{\alpha\varepsilon}^{\, \nu}
$$
and finally, since $N=e^\borne$, we conclude $\varepsilon\in \calE_{\nu}^{(\alpha)}(N)$.\\

{\it Proof of the part {\rm(d)}}. Suppose $d\ge 4$. 
For $\borne>0$, define 
\begin{align} 
\notag
\Dtilde_\nu(\borne)
=& \bigl\{
(x_1,\dots,x_{r+1})\in D_\nu(\borne) \; \mid \; (-x_1,\dots,-x_{r+1})\in D_\nu(\borne) \bigr\}, 
\\[1mm]
\notag 
D''_\nu(\borne) =& \bigl\{
(x_1,\dots,x_{r+1})\in D_\nu(\borne) \; \mid \; (-x_1,\dots,-x_{r+1})\in D'_\nu(\borne) 
\bigr\}, \\[1mm]
\notag 
 \Dtilde'_\nu(\borne) =& \bigl\{
(x_1,\dots,x_{r+1})\in D'_\nu(\borne) \; \mid \; (-x_1,\dots,-x_{r+1})\in D'_\nu(\borne) \bigr\}.
 \end{align} 
If $\Dtilde_\nu(\borne)$ is not empty, then $r_1\ge 4$.
If $D''_\nu(\borne)$ is not empty, then $r_1\ge 2$ and $r_2\ge 1$. 
If $ \Dtilde'_\nu(\borne) $ is not empty, then $r_2\ge 2$.
 
Let us show conversely that if $r_1\ge 4$, then $\Dtilde_\nu(1)$ has a positive volume, 
that if $r_1\ge 2$ and $r_2\ge 1$, then $\Dtilde'_\nu(1)$ has a positive volume and 
that if $r_2\ge 2$, then $\Dtilde'_\nu(1)$ has a positive volume. \\

Let $a$, $b$, $c$ be three positive numbers such that 
$$
\nu<a<b<1\quad \hbox{and}\quad c+2b<2a+1.
$$
For instance
$$
a=\frac{1+\nu}{2}, \quad 
b=\frac{3+\nu}{4}, \quad 
c=\frac{1+\nu}{4}\cdotp
$$ 
Assume $r_1\ge 4$, hence $\delta_1=\delta_2=\delta_3=\delta_4=1$. Then $\Dtilde_\nu(1)$ 
contains the set of $(x_1,x_2,\ldots,x_{r+1})\in H$ verifying
$$
a\le x_1,x_2 \le b, \quad
-b\le x_3,x_4 \le -a, \quad
 \frac{-c}{\delta_i(r-4)}\le x_i\le \frac{c}{\delta_i(r-4)} \quad (5\le i\le r),
$$
because these bounds, together with the equation $\delta_1x_1+\delta_2x_2+\cdots+\delta_{r+1}x_{r+1}=0$ of $H$, imply
$$
-1\le x_{r+1} \le 1. 
$$ 
This shows that $\Dtilde_\nu(1)$ has a positive volume.

Assume $r_1\ge 2$ and $r_2\ge 1$, hence $\delta_1=\delta_2=1$ and $\delta_{r+1}=2$. Then $\Dtilde_\nu(1)$ 
contains the set of $(x_1,x_2,\ldots,x_{r+1})\in H$ verifying
$$
a\le x_1,x_2 \le b, \quad
-b\le x_{r+1} \le -a, \quad
 \frac{-c}{\delta_i(r-3)}\le x_i\le \frac{c}{\delta_i(r-3)} \quad (3\le i\le r-1),
$$
because these bounds and the equation $\delta_1x_1+\delta_2x_2+\cdots+\delta_{r+1}x_{r+1}=0$ of $H$, imply
$$
-1\le x_{r} \le 1. 
$$ 
Therefore $\Dtilde'_\nu(1)$ has a positive volume.


Finally, assume $r_2\ge 2$, hence $\delta_r=\delta_{r+1}=2$. 
Then $\Dtilde'_\nu(1)$ 
contains the set of $(x_1,x_2,\ldots,x_{r+1})\in H$ verifying 
$$
a\le x_{r+1}\le b,\quad -b\le x_r\le -a,\quad
\frac{-c}{\delta_i(r-2)}\le x_i \le \frac{c}{\delta_i(r-2)} \quad (2\le i\le r-1),
$$
because these bounds, together with the equation $\delta_1x_1+\delta_2x_2+\cdots+\delta_{r+1}x_{r+1}=0$ of $H$ imply $-1\le x_1 \le 1$. Hence the $r$--dimension volume of
$\Dtilde'_\nu(1)$ is positive. 

Since $d\ge 4$, in all cases the volume of 
$\Dtilde_\nu(\borne)\cup\Dtilde'_\nu(\borne)\cup\Dtilde'_\nu(\borne)$ is bounded below by an effectively computable positive constant times $\borne^r$. The number of elements in the intersection of this set
with $\ulambda(\alpha)+\ulambda(\ZKtimes)$ is bounded below by an effectively computable positive constant times $\borne^r$.\\
 
Let $\varepsilon\in\ZKtimes$ be such that $\ulambda(\alpha\varepsilon)\in \Dtilde_\nu(\borne)\cup\Dtilde'_\nu(\borne)\cup\Dtilde'_\nu(\borne)$. We have 
$$
\log \max_{1\le j\le d} |\sigma_j(\alpha\varepsilon)|\le \borne, \quad
\log \min_{1\le j\le d} |\sigma_j(\alpha\varepsilon)|\ge -\borne 
 $$
and there exist four distinct elements $\varphi_1,
\varphi_2, \varphi_3, \varphi_4$ of $\Phi$ such that 
 $$
\log|\varphi_i(\alpha\varepsilon)|\ge \nu \borne \quad (i=1,2)
\quad\hbox{and}\quad
 \log|\varphi_j(\alpha\varepsilon)|\le -\nu \borne \quad (j=3,4).
$$
Consequently 
$$
\house{\alpha\varepsilon}\le e^\borne,\quad 
\house{\alpha}^{\, \nu}\le 
|\varphi_i(\alpha\varepsilon)|\le \house{\alpha} \quad (i=1,2) 
$$
and 
$$ 
\house{(\alpha\varepsilon)^{-1}} \le e^\borne,
\quad
\house{(\alpha\varepsilon)^{-1}}^{\,- 1}
\le |\varphi_j(\alpha\varepsilon)|\le \house{(\alpha\varepsilon)^{-1}}^{\,- \nu}
 \quad (j=2,3), 
$$
whereupon finally $\varepsilon\in \calEtilde_{\nu}^{(\alpha)}(e^\borne)$.\\

The part
{\rm(d)} of Proposition $\ref{proposition:densitepositive}$ is then proved.

\section*{Acknowledgements}
The second author thanks the ASSMS (Abdus Salam School of Mathematical Sciences) of Lahore for a fruitful stay in October 2011. 

 \vfill
 
\hbox{
\small
\vbox{
\hbox{Claude LEVESQUE}
	\hbox{D\'{e}partement de math\'{e}matiques et de statistique
	}
	\hbox{Universit\'{e} Laval
	}
	\hbox{Qu\'{e}bec (Qu\'{e}bec)
	}
	\hbox{CANADA G1V 0A6
	}
	\hbox{Claude.Levesque@mat.ulaval.ca
	}
}	
\hfill
\vbox{	\hbox{Michel WALDSCHMIDT 
	}
	\hbox{Universit\'{e} Pierre et Marie Curie (Paris 6),
	}
	\hbox{Th\'{e}orie des Nombres Case courrier 247
	}
	\hbox{4 Place Jussieu 
	}
	\hbox{F -- 75252 PARIS Cedex 05, FRANCE 
	}
	\hbox{miw@math.jussieu.fr}
}	
}	

\vfill

 \end{document}